\documentclass{amsart}
\usepackage{mathptmx}      
\usepackage{graphics}
\usepackage{amstext,amsfonts,amsmath}
\usepackage{amssymb,amscd}
\usepackage{latexsym}

\usepackage{graphicx}

\newtheorem{theorem}{Theorem}[section]
\newtheorem{lemma}[theorem]{Lemma}
\newtheorem{corollary}[theorem]{Corollary}

\newtheorem{claim}{Claim}
\newtheorem{question}[theorem]{Question}

\theoremstyle{definition}

\theoremstyle{remark}

\numberwithin{equation}{section}

\begin{document}

 \title[Siegel's theorem and the Jacobian conjecture over the rational field]{Siegel's theorem on integral points\\ and\\ the Jacobian conjecture over the rational field}

\author{Nguyen Van Chau}
\address{Institute of Mathematics, Vietnam Academy of Science and Technology, 18 Hoang Quoc Viet, 10307 Hanoi, Vietnam.}

\email{nvchau@math.ac.vn}
\thanks{The author was partially supported by Vietnam National Foundation for Science and Technology
Development (NAFOSTED) grant 101.04-2014.23, and VIASM.}

\subjclass[2010]{14R15; 11G05 }
\date{}

\maketitle

\begin{abstract}   It is shown  that  a polynomial map $(P,Q)\in \mathbb{Q}[x,y]^2$
 with $P_xQ_y-P_yQ_x \equiv 1$  has an  inverse map in $\mathbb{Q}[x,y]^2$  if  the  fiber  $P=0$ contains an infinite subset of  $ d^{-1}\mathbb{Z}^2$ for an integer $d$.
  \end{abstract}

\section{Introduction} A polynomial map $F\in \mathbb{C}[X]^n$, $X=(X_1,X_2,\dots, X_n)$, is a {\it Keller map} if it satisfies the Jacobian condition $\det DF\equiv 1$. The mysterious Jacobian conjecture, firstly posed by Ott-Heinrich Keller  \cite{Keller} since 1939 and still opened,  asserts that every Keller  map $F\in \mathbb{C}[X]^n$  has an inverse map in $\mathbb{C}[X]^n$ (see  \cite{EssenBook} and  \cite{Bass}). This  paper is to present a simple application of Siegel's theorem on integral points on affine curves to this conjecture over the rational field.

 Recall that a subset of $\mathbb{Q}^n$ is  {\it quasi-integral}  if  it is contained in  $d^{-1} \mathbb{Z}^n$ for an integer $d$. Obviously, if a Keller map $F=(F_1,F_2,\dots, F_n)\in \mathbb{Q}[X]^n$ has an inverse $G\in \mathbb{Q}[X]^n$,  for each $\alpha=(\alpha_1,\alpha_2,\dots,\alpha_{n-1})\in \mathbb{Q}^{n-1}$, the image $G(\{\alpha\}\times \mathbb{Z})$  is an infinite quasi-integral set contained in  the affine curve  defined by $F_i=\alpha_i$, $i=1,2,\dots, n-1$, where $d$ is  the common denominator of all $\alpha_i$ and coefficients in $G$.

Our main result here is the following.

  \begin{theorem}[Main Theorem]\label{Mainthm}  Let  $(P,Q)\in \mathbb{Q}[x,y]^2$ be a Keller map. If the fiber   $P=0$  contains an infinite quasi-integral subset of  $\mathbb{Q}^2$, then $(P,Q)$ has an inverse map  in $ \mathbb{Q}[x,y]^2$.
\end{theorem}

An important and  immediate consequence of Theorem \ref{Mainthm} with together the formal inverse function theorem is the following.

 \begin{theorem}\label{Thm2}  Every Keller map  $(P,Q)\in \mathbb{Z}[x,y]^2$ has an inverse map  in $ \mathbb{Z}[x,y]^2$ if,  for an $\alpha\in \mathbb{Q}$,  the fiber   $P=\alpha$  contains an infinite quasi-integral subset of $\mathbb{Q}^2$.
 \end{theorem}

Let us denote $$C(S,H,\alpha):=\{(a,b)\in S\times S: H(a,b)=\alpha\}$$
 for $H\in \mathbb{Q}[x,y]$ , $S\subset  \mathbb{Q}$  and $\alpha\in \mathbb{Q}$. In view of Theorem \ref{Mainthm},   for any possible counterexample $(P,Q)\in \mathbb{Q}[x,y]^2$ to the Jacobian conjecture, if exists,  the inequality
  $$ \#C(d^{-1}\mathbb{Z}^2,P,\alpha)\ <+ \infty$$
 must holds true  for all $d\in \mathbb{N}$ and all $\alpha \in \mathbb{Q}$.

 Theorem \ref{Thm2} is a slight improvement of the main result in  \cite{Chau2006}, which says that Keller maps $(P,Q)\in\mathbb{Z}[x,y]^2$ with  fiber $P=0$ having infinitely many integral points  are automorphisms of $\mathbb{Z}^2$. This result is reduced from an interesting observation that  if  a Keller map $(P,Q)\in \mathbb{Z}[x,y]^2$ is not inverse, then  there is a constant $M>0$ depended only on $(P,Q)$  such that
 $\#C(\mathbb{Z},P,k)\leq M$ for all  $k\in \mathbb{Z}$
(Lemma 2,  \cite{Chau2006}).  Our approach here does not cover this result.

In studying  the Jacobian conjecture over the rational field $\mathbb{Q}$ it is worthy to consider the following questions for Keller maps $(P,Q)\in \mathbb{Q}[x,y]^2$:
\begin{question}\label{Q1}Is $(P,Q)$ inverse if  $\#C( \mathbb{Q},P,0)=+\infty$?\end{question}
\begin{question}\label{Q2}Is  uniformly bounded the numbers $\#C( \mathbb{Q},P,\alpha)$, $\alpha\in \mathbb{Q}$,  if $(P,Q)$ is not inverse?\end{question}

Under the Jacobian condition $\det D(P,Q)\equiv 1$,  the complex fibers of $P$ are nonsingular  curves and $P$ is a primitive polynomial in $\mathbb{C}[x,y]$. It is known that for all  except a finite number of  $\alpha\in \mathbb{C}$,  the fibers $P=\alpha$ are diffeomorphic to same a Riemann surface of genus $g_P$ and of $n_P$ punctures. In view of the celebrated Faltings theorem \cite{Faltings}  on rational points on algebraic curves,  if  $\#C( \mathbb{Q},P,0)=\infty$, the fiber $P=0$ must contain  a rational curve or an elliptic curve. Furthermore, if $g_P\geq 2$,  one has $\#C( \mathbb{Q},P,\alpha)<+\infty$ for all except a finite number of $\alpha\in \mathbb{Q}$.

The Uniform Bound Conjecture (see, for example, in  \cite{Caporaso}) says that  for every
integer $ g\geq 2$, there exists a natural number $B( \mathbb{Q}; g)$ such that
any algebraic curve  defined  over $ \mathbb{Q}$ and of genus $g$  cannot have more than $B(\mathbb{Q},g)$
 points in $ \mathbb{Q}^2$. If this conjecture is true and if Question \ref{Q1} is positive,
 then one has a confirmation to Question 1.4,  at least for the case $g_P\geq 2$.
A confirmation to Question \ref{Q2} will allow us to reduce the Jacobian problem
 over $\mathbb{Q}$  to the question {\it whether there is no Keller maps
$(P,Q)\in\mathbb{Q}[x,y]^2$ such that $\#C(\mathbb{Q},P,\alpha)$, $\alpha\in \mathbb{Q}$, are uniformly bounded.}

A proof of Theorem \ref{Mainthm} will be presented in the next section.  A version of this theorem  for high dimensions will be provided in the last section.

 \section{Proof of  Main Theorem}
 Let us begin with a brief introduction on the celebrated Siegel's theorem on integral points on affine curves.
Let $C$ be an irreducible affine curve in $\mathbb{C}^n$ defined by some polynomials in  $\mathbb{Q}[X]$ and $g_C$ denote the geometric genus of a desingularization of $C$.  Siegel's theorem \cite{Siegel}  asserts that {\it if  $g_C>0$  or  if $g_C=0$ and $C$ has more than two irreducible branches at infinity, then $C$ may have at most finitely many integral points}. We will  use the following version of  Siegel's theorem, concerning with affine curves having infinitely many integral points.

 \begin{theorem}\label{Siegelthm}  If  $C$ has infinitely many integral points, then
 \begin{enumerate}
 \item[i)] $C$ has genus zero and has no more than two irreducible branches at infinity, and
 \item[ii)] on each irreducible branch at infinity of $C$, there is a sequence of  integral points of $C$ tending to infinity.
 \end{enumerate}
 \end{theorem}
 Property (i) is just Siegel's theorem stated in an equivalent statement. Property (ii) is known later due to  Silverman  \cite{Silverman}. This property ensures that, in some sense,  the behavior at infinity of a regular function on  such curve $C$ is completely reflected on its restriction on  the set of integral points of $C$. The consequence below may be well-known for experts and  appear somewhere.

 \begin{corollary}\label{Cor1} Let  $C$ be an irreducible affine curve in $ \mathbb{C}^n$,  defined over $\mathbb{Q}$.  Assume that $C$ has an infinite quasi-integral subset. Then, for any $H\in \mathbb{Q}[X]$, the restriction $H_{|C}:C\longrightarrow \mathbb{C}$ of $H$ on $C$ is either a constant function or a proper function. In particular, if $C$ is smooth and $H_{|C}$ has no singularities, then $H_{|C}$ is an isomorphism of $C$ and $\mathbb{C}$.
\end{corollary}
\begin{proof}  Let $H\in\mathbb{Q}[X]$ be fixed.  By assumptions, there is a number  $d\in \mathbb{N}$ such that the intersection $C\cap (d^{-1}\mathbb{Z}^n)$ is infinite and $H\in  d^{-1}\mathbb{Z}[X]$. So, by changing variables $X\mapsto dX$ and $H\mapsto dH$, we can assume that $C$ has infinitely many integral points and $H\in \mathbb{Z}[X]$.

 First, assume that $H$ is not constant on $C$. We will prove that the restriction $H_{|C}:C\longrightarrow \mathbb{C}$ is proper. Observe that by Property (ii) in Theorem \ref{Siegelthm} it suffices to show that  for each sequence  of integral points $a_i\in C$ tending to $\infty$, the corresponding sequence  $H(a_i)$ must tend to $\infty$. To see it, assume the contrary that  $H$ is bounded on a subsequence of $a_i$s. Since any bounded subset of $\mathbb{Z}$ is finite,  $H$  must be a constant on an infinite subset of $\{a_i\}$. This implies that $H$ is constant on $C$ - a contradiction. Hence, $H_{|C}$ is proper.

Now, assume that $C$ is smooth and $H_{|C}$ has no singularities. Since $H_{|C}$ is proper,  $H_{|C}:C\longrightarrow \mathbb{C}$   determines a unramified  covering of $\mathbb{C}$. Thus, by the simple connectedness of $\mathbb{C}$, $H_{|C}$ is isomorphic.
\end{proof}

\begin{lemma}\label{C-curve} Let $(P,Q)\in \mathbb{C}[x,y]^2$ be a Keller map. If the fiber $P=0$ has a component diffeomorphic to  $\mathbb{C}$, then $(P,Q)$  is inverse.\end{lemma}
\begin{proof} Assume that $C$ is  a component of the fiber $P=0$, diffeomorphic to $\mathbb{C}$. As $J(P,Q)\equiv 1$,  the restriction $Q_{|C}: C\longrightarrow \mathbb{C}$ gives a unramified  covering of $\mathbb{C}$, and hence, is bijective. It implies that the restriction $(P,Q)_{|C}=(0,Q_{|C})$ is injective.  By Abhyankar-Moh-Suzuki embedding theorem \cite{Abhyankar} , $C$ is a line in a suitable algebraic coordinate of $\mathbb{C}^2$. The invertibility of $(P,Q)$ now follows from a well-known result due to Gwrozdiewicz \cite{Gwo}, which asserts  that every Keller map of $\mathbb{C}^2$ is inverse if its restriction to a line is injective (see Theorem 1 in  \cite{Gwo},  Theorem 10.2.31 in \cite{EssenBook}).
\end{proof}

\begin{proof}[{\bf Proof of Theorem \ref{Mainthm}}] Let $(P,Q)\in \mathbb{Q}[x,y]^2$ be a given Keller map such that the fiber $P=0$ contains an infinite quasi-integral set of $\mathbb{Q}^2$. In view of Siegel's theorem, the fiber $P=0$ must contain an irreducible component $C$ of genus zero and at most two irreducible branches at infinity. Since $J(P,Q)\equiv 1$, $C$ is smooth and the restriction $Q_{|C}$ has no singularities. Therefore, by Corollary 2.2, the component $C$ is diffeomorphic to $\mathbb{C}$. Hence, by Lemma \ref{C-curve}, $(P,Q)$ has an inverse map in  $\mathbb{Q}[x,y]^2$.
\end{proof}

\section{High dimensional case}

Recall that a  value $c\in \mathbb{C}^m$ is  a {\it generic value} of a polynomial map $h:\mathbb{C}^n\longrightarrow\mathbb{C}^m$, $n\geq m$,  if  there is  an open neighborhood  $U$ of  $c$ such that  the restriction $h: h^{-1}(U)\longrightarrow U$  determines a locally trivial  fibration.  Let $E_h$ denote the  the complement of the set of all generic values of $h$.  By definitions, the restriction $h:\mathbb{C}^n\setminus h^{-1}(E_h)\longrightarrow\mathbb{C}^m\setminus E_h$
 determines a locally trivial fibration.  It is well-known that either $E_h$  is empty and $h$ is a trivial fibration or $E_h$ is an algebraic hypersurface of  $\mathbb{C}^m$  (for example, see  \cite{Varchenko}).

Our version of Theorem \ref{Mainthm} for high dimensional cases  can be stated as follows.
\begin{theorem}\label{Thm3} Let
$F=(F_1,F_2,\dots , F_n) \in \mathbb{Q}[X]^n$ be a Keller map. Assume that  there is  a generic value $\alpha\in \mathbb{Q}^{n-1}$  of the map $\hat F=(F_1,F_2,\dots,F_{n-1}): \mathbb{C}^n\longrightarrow \mathbb{C}^{n-1}$ such that  the fiber $\hat F=\alpha$ contains an infinite   quasi-integral set of $\mathbb{Q}^n$. Then, $F$ has an inverse map in $\mathbb{Q}[X]^n$.
\end{theorem}

In view of Siegel's theorem, the assumption on the fiber $\hat F=\alpha$ ensures that the generic fiber of $\hat F$ is diffeomorphic to either $\mathbb{C}$ or $\mathbb{C}^*$. Theorem \ref{Thm3} is an immediate consequence of the formal inverse function theorem and the following lemma.

\begin{lemma} Let
$F=(F_1,F_2,\dots , F_n) \in \mathbb{C}[X]^n$ be a Keller map and  $\hat F:=(F_1,F_2,\dots,F_{n-1}): \mathbb{C}^n\longrightarrow \mathbb{C}^{n-1}$. Then,
\begin{enumerate}
\item[a)]  $F$ is inverse if the generic fiber of $\hat F$ is diffeomorphic to $\mathbb{C}$;
\item[b)] the generic fiber of $\hat F$ can never be diffeomorphic to $\mathbb{C}^*$.
\end{enumerate}

\end{lemma}
\begin{proof}For $\lambda\in \mathbb{C}^{n-1}$, let us denote $C_\lambda:=\hat F^{-1}(\lambda)$, $L_\lambda:=\{\lambda\}\times \mathbb{C}$ and $f_\lambda:C_\lambda \longrightarrow L_\lambda\subset \mathbb{C}^n$ the restriction to $C_\lambda$ of $F$, $f_\lambda(a)=(\lambda, F_n(a))$,  $a\in C_\lambda$. Since $JF\equiv 1$, the fibers $C_\lambda$ are smooth and  the maps  $f_\lambda$ have no singularities.  By definitions, the map
$\hat F:\mathbb{C}^n\setminus \hat F^{-1}(E_{\hat F})\longrightarrow\mathbb{C}^{n-1}\setminus E_{\hat F}$
 is a locally trivial fibration. So, the fibres $C_\lambda$,  $\lambda\in  \mathbb{C}^{n-1}\setminus E_{\hat F}$, are nonsingular irreducible affine curves of  same a topological type.

a) Assume that  the generic fiber of  $\hat F$ is diffeomorphic to the line $\mathbb{C}$. Then, for every $\lambda\in  \mathbb{C}^{n-1}\setminus E_{\hat F}$, the  map $f_\lambda:C_\lambda \longrightarrow L_\lambda\subset \mathbb{C}^n$ is  diffeomorphic. It follows that $\# F^{-1}(a)= 1$ for all points $a$ of the  open dense  algebraic subset $\mathbb{C}^n\setminus (E_{\hat F}\times \mathbb{C})$ of $\mathbb{C}^n$. Since  $F$ is locally diffeomorphic by the Jacobian condition, it follows that  $F$ is injective. Hence, by Ax-Grothendieck Theorem (Theorem 10.4.11 in \cite{Gro}, see also \cite{Ax, Kenna}), $F$ is inverse.

b) Assume  the contrary  that $C_\alpha$ is diffeomorphic to $\mathbb{C}^*$. Then,  $F$ is not inverse and the sets $E_F$ and $E_{\hat F}$ are hypersurfaces of $\mathbb{C}^n$ and $\mathbb{C}^{n-1}$, respectively.

First, we will show that  for each $\lambda \in \mathbb{C}^{n-1}\setminus E_{\hat F}$ there is a $b_\lambda\in \mathbb{C}$ such that
 $F^{-1}(\lambda,b_\lambda)=\emptyset$ and
 the map
 $$f_\lambda: C_\lambda\longrightarrow L_\lambda\setminus\{(\lambda,b_\lambda)\}\eqno(*)$$
gives a unramified covering.

\medskip
Let $\lambda \in \mathbb{C}^{n-1}\setminus E_{\hat F}$ be fixed. Let $\Gamma_1$ and $\Gamma_2$ be the two unique irreducible branches at infinity of $C_\lambda$ and $b_1$ and $b_2$ be the corresponding limiters of sequences $F_n(a_k)$ where $a_k\in \Gamma_i$ tend to infinity. Observe that at least one of $b_i$ is $\infty$. Otherwise, if both of $b_i$ are $\infty$, $f_\lambda$ must be proper, and hence, must be a diffeomorphism from  $C_\lambda$ onto $L_\lambda$. Thus, we can assume that $b_1=\infty$ and $b_2:=b_\lambda\in \mathbb{C}$.

Consider the covering  $$f_\lambda: C_\lambda\setminus F^{-1}(\lambda,b_\lambda)\longrightarrow L_\lambda\setminus\{(\lambda,b_\lambda)\}.$$
 By applying the Riemann-Huzwicz relation we have
  $$\chi(C_\lambda\setminus F^{-1}(\lambda,b_\lambda))=\deg_{geo.}f_\lambda . \chi(L_\lambda\setminus\{(\lambda,b_\lambda)\})-\sum_{p\in C_\lambda\setminus F^{-1}(\lambda,b_\lambda)} \deg_p f_\lambda-1.$$
Here, $\chi(V)$,  $\deg_{geo.}f_\lambda$ and $\deg_pf_\lambda$  denote the Euler-Poincare characteristic of an affine curve $V$, the geometric degree of $f_\lambda$ and the local degree of $f_\lambda$ at $p\in C_\lambda$, respectively.
Note that $\chi(L_\lambda\setminus\{(\lambda,b_\lambda)\}))=0$ and  $f_\lambda$ has no singularities. From the above equality it follows that
$\chi(C_\lambda\setminus F^{-1}(\lambda,b_\lambda))=0$.  Since $C_\lambda$ is diffeomorphic to $\mathbb{C}^*$,  the inverse image $F^{-1}(\lambda,b_\lambda)$ is just empty and  the covering (*)  is  unramified.
\medskip

Now,  let  $E_0:=\{a\in \mathbb{C}^n: F^{-1}(a)=\emptyset\}$ and consider the projection $\pi: E_0\longrightarrow \mathbb{C}^{n-1}$, $\pi(X):=(X_1,X_2,\dots, X_{n-1})$. Note that $E_0$ is a closed algebraic subset of $\mathbb{C}^n$. By the previous claim the restriction $\pi: \pi^{-1}(\mathbb{C}^{n-1}\setminus E_{\hat F})\longrightarrow \mathbb{C}^{n-1}\setminus E_{\hat F}$ is one-to-one. It follows that  the inverse image  $\pi^{-1}(\mathbb{C}^{n-1}\setminus E_{\hat F})$, which is an open algebraic subset of $E_0$,  is of dimension $n-1$. Therefore,   $E_0$ contains a hypersurface of $\mathbb{C}^n$. This is impossible. Indeed, if  $E_0$ contains a hypersurface defined by a nonconstant polynomial $H\in \mathbb{C}[X]$,  it must be that $H( F(X))\equiv c\neq 0$. Therefore, $DH(F(X))DF(X)\equiv 0$ that contradicts to the Jacobian condition.
\end{proof}

\medskip

\begin{center} ***\end{center}

\medskip
To conclude the article, we would like to present some remarks related to Theorem \ref{Mainthm} and Theorem \ref{Thm3}.

i) Siegel's theorem is stated and valid for  number fields.  Property (ii) in Theorem \ref{Siegelthm} is also valid for an arbitrary number field. Its proof is implicit  in the proof of the main results in  \cite{Poularkis2002} and in the algorithms finding integral points  in  \cite{Poularkis2011}.

ii) As seen in its proof, Theorem \ref{Thm3}  still holds true for  when $\mathbb{Q}$ is replaced by an arbitrary number field.

iii) In our arguments to prove Corollary \ref{Cor1} the only fact on the integral  ring $\mathbb{Z}$ of $\mathbb{Q}$ used is that any bounded subset of $\mathbb{Z}$ is finite. This is true for  integral rings of imaginary quadratic fields $\mathbb{Q}(\sqrt{-m})$, $m\in \mathbb{N}$. Thus, Theorem 1.1 is valid for the fields $\mathbb{Q}$ and $\mathbb{Q}(\sqrt{-m})$, $m\in \mathbb{N}$.  The analogous statement of this theorem for number fields would be  true if one could prove that for Keller maps $(P,Q)\in \mathbb{C}[x,y]^2$,   fibers of $P$ cannot have  components diffeomorphic to $\mathbb{C}^*$.

\bigskip

\noindent{\it Acknowledgments:} The author wishes to
express his thank to the professors Arno Van den Essen, Ha Huy Vui and  Pierrette Cassou-Nogu\`es for valuable discussions. The author would like to thank the VAST-JSPS   and the VIASM for their helps.

\end{document}